\documentclass[12pt]{amsart}

\usepackage{amsmath,amscd,amssymb,amsopn,amsmath,amsthm,mathrsfs,graphics,amsfonts,enumerate,verbatim,calc}
\usepackage[all]{xy}

\topmargin = - 0.7cm
\usepackage{color}

\usepackage[OT2,OT1]{fontenc}
\newcommand\cyr{%
\renewcommand\rmdefault{wncyr}%
\renewcommand\sfdefault{wncyss}%
\renewcommand\encodingdefault{OT2}%
\normalfont
\selectfont}
\DeclareTextFontCommand{\textcyr}{\cyr}

\DeclareFontFamily{OT1}{rsfs}{}
\DeclareFontShape{OT1}{rsfs}{n}{it}{<-> rsfs10}{}
\DeclareMathAlphabet{\mathscr}{OT1}{rsfs}{n}{it}

\numberwithin{equation}{section}
\newtheorem{theorem}{Theorem}[section]
\newtheorem{lemma}[theorem]{Lemma}
\newtheorem{proposition}[theorem]{Proposition}
\newtheorem{corollary}[theorem]{Corollary}

\theoremstyle{definition}
\newtheorem{definition}[theorem]{Definition}
\newtheorem{remark}[theorem]{Remark}
\newtheorem{question}[theorem]{Question}
\newtheorem{problem}[theorem]{Problem}

\newtheorem{example}[theorem]{Example}

\newcommand{\rel}{\operatorname{rt}}

\newcommand{\Spec}{\operatorname{Spec}}

\newcommand{\gr}{\operatorname{gr}}

\newcommand{\Ext}{\operatorname{Ext}}

\newcommand{\Ann}{\operatorname{Ann}}

\newcommand{\hdeg}{\operatorname{hdeg}}

\newcommand{\Proj}{\operatorname{Proj}}

\newcommand{\ini}{\operatorname{in}}
\newcommand{\arn}{\operatorname{ar}}
\newcommand{\e}{\varepsilon}
\newcommand{\D}{\delta}
\newcommand{\reg}{\operatorname{reg}}
\newcommand{\sdeg}{\operatorname{sdeg}}

\newcommand{\fm}{\frak{m}}
\newcommand{\fp}{\frak{p}}

\begin{document}
\title{When does a perturbation of the equations\\ preserve the normal cone?}

\author[Pham Hung Quy]{Pham Hung Quy}
\address{Department of Mathematics, FPT University, Hanoi, Vietnam}
\email{quyph@fe.edu.vn}

\author[Ngo Viet Trung]{Ngo Viet Trung}
\address{International Centre for Research and Postgraduate Training, Institute of Mathematics, Vietnam Academy of Science and Technology, 18 Hoang Quoc Viet, Hanoi, Vietnam}
\email{nvtrung@math.ac.vn}


\keywords{Perturbation,  singularity, normal cone, filter regular sequence, associated graded ring, Rees algebra, initial ideal, Artin-Rees number, Hilbert-Samuel function, extended degree, filtration, monomial order}
\subjclass[2010]{Primary 13A30, 13D10; Secondary 13H15, 13P10, 14B12}

\begin{abstract} 
Let $(R,\fm)$ be a local ring and $I, J$ two arbitrary ideals of $R$. Let $\gr_J(R/I)$ denote the associated graded ring of $R/I$ with respect to $J$, which corresponds to the normal cone in algebraic geometry. With regards to the finite determinacy of singularity with respect to the Jacobian ideal \cite{CS1,CS2} we study the problem for which ideal $I = (f_1,...,f_r)$ does there exist a number $N$ such that if $f_i' \equiv f_i \mod J^N$, $i = 1,...,r$, and  $I' = (f_1',...,f_r')$, then $\gr_J(R/I) \cong \gr_J(R/I')$. This problem arises from a recent result of Ma, Quy and Smirnov in the case $J$ is an $\fm$-primary ideal \cite{MQS}, which solves a long standing conjecture of Srinivas and Trivedi on the invariance of Hilbert functions under small perturbations \cite{ST1}. Their approach involves Hilbert functions and cannot be used to study the above general problem. 
Our main result shows that such a number $N$ exists if $f_1,...,f_r$ is locally a regular sequence outside the locus of $J$. It has interesting applications to a range of related problems. 
\end{abstract}

\maketitle

\section{Introduction}

Let $I = (f_1,...,f_r)$ be an ideal in a local ring $(R,\fm)$. An ideal $I' = (f_1',...,f_r')$ is called
a small perturbation of $I$ if $f_i' \equiv f_i  \mod \fm^N$ for $N \gg 0$. It is of great interest to know which properties of $R/I$ are preserved by $R/I'$ under small perturbation.
For instance, when studying the singularities of an analytic space we often want to replace the defining power series by their $n$-jets for some $n \gg 0$. This problem have been studied in singularity theory, mainly with the aim to find necessary or sufficient conditions for $R/I \cong R/I'$ (see e.g. \cite{CS1,CS2,GP,MS}). \par

In 1996, Srinivas and Trivedi \cite{ST1} showed that the Hilbert-Samuel function of $R/I$ with respect to an $\fm$-primary ideal $J$ does not change under small perturbations of $I$ if $R$ is a generalized Cohen-Macaulay ring and $f_1,...,f_r$ is part of a system of parameters.  They conjectured that the same is true if $R$ is an arbitrary local ring and $f_1,...,f_r$ is a filter-regular sequence, which means that $f_1,...,f_r$ is locally a regular sequence on the punctured spectrum. The conjecture of Srinivas and Trivedi has been recently solved by Ma, Quy and Smirnov \cite{MQS}. Their proof is based on two deep results of Huneke and Trivedi on the height of ideals \cite{HT} and of Eisenbud on approximation of complexes \cite{Ei}.
Actually, Srinivas and Trivedi \cite{ST1} as well as Ma, Quy and Smirnov \cite{MQS} proved that $\gr_J(R/I) \cong \gr_J(R/I')$, which implies that the Hilbert-Samuel functions of $R/I$ and $R/I'$ with respect to $J$ are the same. 

The work of Ma, Quy and Smirnov \cite{MQS} leads to the following general problem:

\begin{problem} \label{general} 
Let $J$ be an {\em arbitrary ideal}. For which ideal $I = (f_1,...,f_r)$ does there exist a number $N$ such that if $f_i' \equiv f_i \mod J^N$, $i = 1,...,r$, and  $I' = (f_1',...,f_r')$, then $\gr_J(R/I) \cong \gr_J(R/I')$?
\end{problem}

This general problem is of interest because $\Spec(\gr_J(R/I))$ is the {\em normal cone} and $\Proj(\gr_J(R/I))$ is the exceptional fiber of the blow-up of $R/I$ along $J$. 
The problem of finding conditions for $R/I \cong R/I'$ has been already considered in singularity theory, when $R$ is a power series ring and $J$ is  the Jacobian ideal of $I$, by Cutkosky and Srinivasan \cite{CS1,CS2}, M\"ohring and van Straten \cite{MS}.  The method of Ma, Quy and Smirnov can not be used to study Problem \ref{general} because it depends on the theory of Hilbert-Samuel functions, which is not available when $J$ is not an $\fm$-primary ideal. 

We have found a new approach which allows us to solve  
Problem \ref{general}. Our main result is the following theorem.

\begin{theorem} \label{intro}
Let $J$ be an arbitrary ideal of a local ring $R$. Let $I = (f_1,...,f_r)$, where $f_1,...,f_r$ is a $J$-filter regular sequence.  
There exists an explicit number $N$ such that if $f_i' \equiv f_i \mod J^N$, $i = 1,...,r$, and $I' = (f_1',...,f_r')$, then $f_1',...,f_r'$ is a $J$-filter regular sequence and
$\gr_J(R/I) \cong \gr_J(R/I').$
\end{theorem}

Recall that a sequence of elements $f_1,...,f_r$ in $R$ is  {\em $J$-filter regular} if $f_i \not\in \fp$ for any associated prime $\fp \not\supseteq J$ of $(f_1,...,f_{i-1})$, $i = 1,....,r$, i.e. $f_1,...,f_r$ is locally a regular sequence outside the locus of $J$.  Besides the usual filter-regular sequences, which
plays an important role in the theory of generalized Cohen-Macaulay rings \cite{SCT}, this notion appears in several topics such as local cohomology modules \cite{ASc,DQ,NS, Sc}, $d$-sequences \cite{Hu,Tr1}, and Macaulayfication \cite{CC,Ka}. 
Associated graded rings of the form $\gr_J(R/I)$, where $I$ is generated by a $J$-filter regular sequence, arise, for instance, in the theory of generalized multiplicities \cite{AM,FOV}. 

The number $N$ of Theorem \ref{intro} is given in terms of Artin-Rees numbers and Loewy lengths attached to the sequence $f_1,...,f_r$ (see Theorem \ref{main} for details). 
This solves a problem of Ma, Quy and Smirnov \cite[Question 1.1]{MQS} which asks for an explicit bound for the number $N$ in their result. This number is important for applications, for instance, if we want to approximate an analytic singularity by an algebraic singularity.

The proof of Theorem \ref{intro} is based on the relationship between initial ideals and Artin-Rees numbers. It allows us to get better insight into the nature of perturbations. Along the proof of Theorem \ref{intro}, we also show that the Artin-Rees numbers of $I$ and $I'$ with respect to $J$ are the same under small perturbations, which confirms another problem raised by Ma, Quy and Smirnov in \cite[Paragraph before Corollary 3.6]{MQS}. 

We also obtain the following result, which provides a converse to Theorem \ref{intro}. Notice that Theorem \ref{intro} implies $\gr_J(R/(f_1,...,f_i)) \cong \gr_J(R/(f_1',...,f_i'))$ for all $i = 1,...,r$.

\begin{theorem} \label{intro2}
Let $J$ be an arbitrary ideal and $f_1,...,f_r$ a sequence of elements in a local ring $R$.
There exists a number $N$ such that 
$$\gr_J(R/(f_1,...,f_i)) \cong \gr_J(R/(f_1',...,f_i'))$$
for any sequence $f_1',...,f_i'$ with  $f'_i \equiv f_i \mod J^N$, $i = 1,...,r$, if and only if $f_1,...,f_r$ is a $J$-filter regular sequence.
\end{theorem}

In particular, $\gr_J(R/(f_1)) \cong \gr_J(R/(f_1'))$ for any  $f_1' \equiv f_1  \mod J^N$, $N \gg 0$, if and only if $f_1$ is a $J$-filter regular element. Theorem \ref{intro2} shows that the condition of $I$ being generated by a $J$-filter regular sequence is practically the best solution to Problem \ref{general}.  

Theorem \ref{intro} has several interesting applications. For a non-$\fm$-primary ideal, there is a bivariate function introduced by Achilles and Manaresi \cite{AM}, which plays a similar role as the Hilbert-Samuel function. It is asymptotically a polynomial function whose normalized leading coefficients give a multiplicity sequence, which generalizes the Hilbert-Samuel multiplicity and the Segre numbers in singularity theory \cite{FOV,GG,PTUV}. We show that the Achilles-Manaresi function of $R/I$ with respect to $J$ does not change under small $J$-adic perturbations of $I$ (Corollary \ref{AM}). Since the structure of $\gr_J(R/I)$ is strongly related to that of the Rees algebra $\Re_J(R/I)$, which corresponds to the notion of blow-up of $R/I$ along $J$ in algebraic geometry, we can also show 
that the relation type, the Castelnuovo-Mumford regularity, the Cohen-Macaulayness and the Gorensteiness of $\Re_J(R/I)$ are preserved under small $J$-adic perturbations of $I$ (Corollary \ref{Rees}). All these results would be considered as very difficult if one does not know Theorem \ref{intro}.

Given an $\fm$-primary ideal $J$, we call the least number $N$ such that if $I' = (f_1',...,f_r')$ with $f_i' \equiv f_i \mod J^N$,  then $R/I$ and $R/I'$ have the same Hilbert-Samuel function with respect to $J$ (or $\gr_J(R/I) \cong \gr_J(R/I')$), the {\em Hilbert perturbation index} of $R/I$ (or the {\em perturbation index} of $\gr_J(R/I)$). In general, an upper bound for the Hilbert perturbation index of $R/I$ can always be used to bound the perturbation index of $\gr_J(R/I)$ (Proposition \ref{Hilbert}). 
If $R$ is a Cohen-Macaulay ring, Srinivas and Trivedi \cite{ST2} gave an upper bound for the Hilbert perturbation index of $R/I$ in terms of the multiplicity. 
If $R$ is not a Cohen-Macaulay ring, one can not bound the Hilbert perturbation index solely in terms of the multiplicity. 
If $R$ is a generalized Cohen-Macaulay ring, Quy and V.D.~Trung \cite{QT} gave an upper bound for the Hilbert perturbation index of $R/I$ in terms of the homological degree, a generalization of the multiplicity. 
However, their proof does not work in the general case, when $R$ is an arbitrary local ring. Using Theorem \ref{intro} we are able to give a general upper bound for the perturbation index of $\gr_J(R/I)$ in terms of any extended degree, which includes the homological degree as a special case (Theorem \ref{m-primary}).

Finally, as an application of our approach we study the more general problem:

\begin{problem} \label{filteration}
Let $F = \{J_n\}_{n\ge 0}$ be a filtration of ideals in $R$. Let $I = (f_1,...,f_r)$. 
When does there exists a number $N$ such that $I$ and $I'$ have the same initial ideal with respect to $F$ for any ideal $I' = (f_1',...,f_r')$ with $f_i \equiv f_i' \mod J_N$?
\end{problem}

We are inspired by a recent conjecture of Adamus and Seyedinejad \cite{AS} which states that if $I = (f_1,...,f_r)$ is a complete intersection ideal in a convergent power series ring and $I' = (f_1',...,f_r')$, where $f_i'$ is the $n$-jet of $f_i$ for some $n \gg 0$, $i = 1,...,r$, then $I$ and $I'$ have the same initial ideals with respect to the degree lexicographic monomial order. A solution to this conjecture can be deduced from the afore mentioned result of Srinivas and Trivedi. A direct proof has been given by Adamus and Patel \cite{AP1,AP2}. In general, the initial ideal with respect to any Noetherian monomial order can be viewed as the initial ideal with respect to a Noetherian filtration. 
Initial ideals with respect to filtrations also appear in Arnold's classification of hypersurface singularities \cite{Ar,BGM}. For an arbitrary Noetherian filtration $F$, we show that there exists a  number $N$ such that $(f_1,...,f_i)$ and $(f_1',...,f_i')$ have the same initial ideal with respect to $F$ for any sequence $f_1',...,f_r'$ with $f_i \equiv f_i' \mod J_N$, $i = 1,...,r$, if and only if $f_1,...,f_r$ is a $J_1$-filter regular sequence (Theorem \ref{main2} and Theorem \ref{converse 2}). This gives a satisfactory answer to Problem \ref{filteration}.

The paper is organized as follows. In Section 2 we prepare basic facts on initial ideal and Artin-Rees number. In Section 3 we prove Theorem \ref{intro} and Theorem \ref{intro2}. Actually, these results follow from Theorem \ref{main} and Theorem \ref{converse}, respectively. In Section 4 we give bounds for the perturbation index in terms of the extended degree. Perturbations with respect to filtrations of ideals are dealt with in Section 5. For unexplained terminology we refer the reader to \cite{BH}.

\noindent{\bf Acknowledgement}. This work is supported by grant NCXS02.01/22-23 of Vietnam Academy of Science and Technology. It was initiated during a research stay of the authors at Vietnam Institute for Advanced Study in Mathematics (VIASM) in 2020. 


\section{Initial ideal and Artin-Rees number}

Let $(R,\fm)$ be a local ring. Let $I$ and $J$ be arbitrary ideal in $R$. Set
$$\gr_J(R/I) := \bigoplus_{n \ge 0} (J^n+I)/(J^{n+1}+I),$$
which is the associated graded ring of $R/I$ with respect to the ideal $J+I/I$.

For every element $f \neq 0$ in $R$ we denote by $o(f)$ the {\em order} of $f$ with respect to the $J$-adic filtration, which is the largest number $n$ such that $f \in J^n$, and by $f^*$ the {\em initial element} of $f$ in $\gr_J(R)$, which is the residue class of $f$ in $J^{o(f)}/J^{o(f)+1}$. 
For convenience, we set $0^* = 0$. Let 
$\ini(I)$ denote the {\em initial ideal} of $I$ in $\gr_J(R)$, which is generated by the elements $f^*$, $f \in I$. 

The following properties of initial ideals are more or less known.

\begin{lemma} \label{repre}
$\gr_J(R/I) \cong \gr_J(R)/\ini(I).$
\end{lemma}

\begin{proof} 
It is easily seen that 
$\ini(I)  = \bigoplus_{n \ge 0} (J^n \cap I + J^{n+1})/J^{n+1}$ for all $n \ge 0$. Therefore,
\begin{align*}
\gr_J(R/I) & = \bigoplus_{n \ge 0} \dfrac{J^n+I}{J^{n+1}+I} 
\cong \bigoplus_{n \ge 0} \dfrac{J^n}{J^n \cap (J^{n+1}+ I)}\\
 & = \bigoplus_{n \ge 0} \dfrac{J^n}{J^n \cap I +J^{n+1}} = \gr_J(R)/\ini(I).
\end{align*}
\end{proof}

\begin{lemma} \label{qin}
Let $K$ be an ideal in $I$. Let $\ini(I/K)$ denote the initial ideal of $I/K$ in $\gr_J(R/K)$. Then
 $\ini(I/K) = \ini(I)/\ini(K)$.
\end{lemma}

\begin{proof} 
We have
\begin{align*}
\ini(I/K) & = \bigoplus_{n \ge 0} \frac{J^n \cap I + J^{n+1}+K}{J^{n+1}+K} 
= \bigoplus_{n \ge 0} \frac{J^n \cap I +J^{n+1}}{(J^n \cap I +J^{n+1})\cap (J^{n+1}+K)}\\
& = \bigoplus_{n \ge 0} \frac{J^n \cap I+J^{n+1}}{J^n \cap K +J^{n+1}} = \ini(I)/\ini(K).
\end{align*}
\end{proof}

By Artin-Rees lemma we know that there exist a number $c$ such that
$$J^n \cap I = J^{n-c}(J^c \cap I)$$
for all $n \ge c$. The least number $c$ with this property is called the {\em Artin-Rees number} of $I$ with respect to $J$. We denote this number by $\arn_J(I)$. It is easily seen that $\arn_J(I)$ is the least number $c$ such that $J^{n+1} \cap I  = J(J^n \cap I)$ for all $n \ge c$. \par
 
For any graded ideal $Q$, we denote by $Q_n$ the $n$-th graded component of $Q$ and by 
$d(Q)$ the maximum degree of the elements of a graded minimal generating set of $Q$.

\begin{proposition}\label{ar} 
$\arn_J(I) = d(\ini(I))$.  
\end{proposition}

\begin{proof} 
Let $S = \gr_J(R)$. Note that $S$ is a standard graded ring, i.e. $S$ is generated by $S_1$ over $S_0$. Then $d(\ini(I))$ is the least integer $c$ such that $\ini(I)_{n+1} = S_1\ini(I)_n$ for all $n \ge c$. We have $\ini(I)_n = (J^n \cap I+J^{n+1})/J^{n+1}.$ Therefore, $d(\ini(I))$ is the least integer $c$ such that
$$J^{n+1} \cap I + J^{n+2} = J(J^n \cap I +J^{n+1}) = J(J^n \cap I)+J^{n+2}$$
for all $n \ge c$. For all $n \ge  \arn_J(I)$, we have 
$J^{n+1} \cap I = J(J^n \cap I).$ 
Hence, $\arn_J(I) \ge d(\ini(I))$.
\par

To prove $\arn_J(I) \le d(\ini(I))$, it suffices to show that $J^{n+1} \cap I = J(J^n \cap I)$ for $n \ge d(\ini(I))$.
Let $f \in J^{n+1} \cap I$. Then $f \in J(J^n \cap I)+J^{n+2}.$
Write $f = g  +  f_1$ for some $g  \in J(J^n \cap I)$ and $f_1 \in J^{n+2}$. Since $f, g  \in I$, we have 
$$f_1 \in J^{n+2} \cap I \subseteq J(J^{n+1} \cap I)+J^{n+3} \subseteq J(J^n \cap I)+J^{n+3}.$$ 
Hence, $f  = g + f_1 \in J(J^n \cap I) + J^{n+3}.$
Continuing like that we have $f \in J(I \cap J^n) + J^{n+t}$ for all $t \ge 1$. 
By Krull's intersection theorem, this implies $f \in J(J^n \cap I)$.  Thus, $J^{n+1} \cap I \subseteq J(J^n \cap I)$. Since $J^{n+1} \cap I \supseteq J(J^n \cap I)$, we get $J^{n+1} \cap I = J(J^n \cap I)$ for $n \ge d(\ini(I))$. The proof is now complete. 
\end{proof}

\begin{remark} 
We can also characterize the Artin-Rees number $\arn_J(I)$ in terms of the Rees algebra
$\Re_J(R) := \oplus_{n \ge 0}J^n$. 
Let $Q = \oplus_{n \ge 0}(J^n \cap I)$, which is a graded ideal of $\Re_J(R)$. Note that $\Re_J(R)$ is a standard graded ring. Then $d(Q)$ is the least number $c$ such that $J^{n+1} \cap I  = J(J^n \cap I)$ for $n \ge c$.
Therefore, $\arn_J(I) = d(Q)$.
\end{remark}

The estimation of the maximal degree of the minimal generators of a graded ideal is in general difficult.
It often leads to the computation of the Castelnuovo-Mumford regularity, which is defined as follows.
 
Let $S$ be a finitely generated standard graded algebra over a commutative ring $A$. Let $H^i_{S_+}(S)$ denote the $i$-th local cohomology module of $S$ with respect to the graded ideal $S_+$ of elements of positive degree and set $a_i(S) := \max\{n| H^i_{S+}(S)_n =0\}$ with the convention $a_i(S) = -\infty$ if $H^i_{S+}(S) = 0$. Then one defines
$$\reg(S) := \max\{a_i(S) + i| i \ge 0\}.$$
We refer the reader to \cite{Tr3} for basic facts on the Castelnuovo-Mumford regularity of standard graded algebras over a commutative ring. Note that other references usually deal with standard graded rings over a field or an artinian local rings.

If we represent $S$ as a quotient ring $A[X]/P$ of a polynomial ring over $A$, then the maximum degree $d(P)$ of a graded minimal generating set of $P$ is an invariant of $S$, which is called the {\it relation type} of $S$ and denoted by $\rel(S)$ \cite{PV}. It follows from \cite[Corollary 2.6]{Tr3} that $\rel(S) \le \reg(S)+1$.

The following bound for the Artin-Rees number was obtained by combining results of Planas-Vilanova on a uniform Artin-Rees property \cite{PV1} and of Ooishi on the Castelnuovo-Mumford regularity of the Rees algebra \cite{O} in \cite[Theorems 2.7 and 2.8]{MQS}. This result can be easily explained by means of initial ideal.

\begin{corollary} \label{reg}
$\arn_J(I) \le \reg(\gr_J(R/I))+1$.
\end{corollary}

\begin{proof}
Let $A = R/J$. Let $\gr_J(R) = A[X]/Q$, where $A[X]$ is a polynomial ring over $A$ and $Q$ a homogeneous ideal in $A[X]$. By Lemma \ref{repre}, $\gr_J(R/I) \cong \gr_J(R)/\ini(I).$ 
Let $P$ be the ideal of $A[X]$ such that $\ini(I) = P/Q$. Then $\gr_J(R/I) = A[X]/P$. By Proposition \ref{ar}, $\arn_J(I) = d(\ini(I)) \le d(P) = \rel(\gr_J(R/I)) \le \reg (\gr_J(R/I))+1$. 
\end{proof}

The following fact will be useful when dealing with the Artin-Rees number of ideals in quotient rings.

\begin{lemma} \label{qar}
Let $K$ be an ideal in $I$. Let $\bar J = (J+K)/K$ and $\bar I = I/K$. Then 
$$\arn_{\bar J}(\bar I) \le \arn_J(I).$$
\end{lemma}

\begin{proof}
For $n \ge \arn_J(I)$ we have 
$$(J^{n+1}+K) \cap I = J^{n+1} \cap I + K = J(J^n \cap I) + K = (J + K)((J^n + K) \cap I) + K,$$
This can be rewritten as $\bar J^{n+1}  \cap \bar I = \bar J(\bar J^n \cap \bar I)$.
Hence, $\arn_{\bar J}(\bar I) \le \arn_J(I)$.
\end{proof}

For a sequence of elements $f_1,...,f_r$ in $R$ we will use the shorter notations $\ini(f_1,...,f_r)$ and $\arn_J(f_1,...,f_r)$ instead of $\ini((f_1,...,f_r))$ and $\arn_J((f_1,...,f_r))$.


\section{Perturbation of filter-regular sequences}

Let $(R,\fm)$ be a local ring and $J$ an arbitrary ideal of $R$.
A sequence of elements $f_1,...,f_r$ in $R$ is called {\em $J$-filter regular} if $f_i \not\in \fp$ for any associated prime $\fp \not\supseteq J$ of $(f_1,...,f_{i-1})$, $i = 1,....,r$. 
This notion is introduced in \cite{Tr1} and has its origin in the theory of generalized Cohen-Macaulay rings \cite{SCT}. If $\sqrt{J} = \fm$, one uses the term filter regular sequence.

An element $f  \in R$ is called $J$-filter regular if $f$ is a $J$-filter regular sequence of one element, i.e.
$f \not\in \fp$ for any associated prime $\fp \not\supseteq J$ of $R$. It is easy seen that this condition is satisfied if and only if there exists a number $n$ such that $J^n(0:f) = 0$. 

Let $I$ be an ideal of $R$. Let $\bar f$ denote the residue class of $f$ in $R/I$ and $\bar J = (J+I)/I$.
For convenience we say that $f$ is a $J$-filter regular element in a quotient ring $R/I$ if $\bar f$ 
is a $\bar J$-filter regular element. It is clear that $f_1,...,f_r$ is a $J$-filter regular sequence if and only if $f_i$ is a $J$-filter regular element in $R/(f_1,...,f_{i-1})$, $i = 1,...,r$. Unlike a regular sequence, $J$-filter regular sequence is not permutable, i.e. it depends on the order of the elements in the sequence.

The following examples give large classes of $J$-filter regular sequences. \medskip

\begin{example} ~\par
\begin{enumerate}[(1)]
\item Let $J = \prod_{i=0}^{d-1}\Ann(H_\fm^i(R))$. Then every system of parameters of $R$ is a $J$-filter regular sequence  \cite[Proposition 4]{Sc}. 
Note that $R$ is called a generalized Cohen-Macaulay ring if 
$\sqrt{J} \supseteq \fm$. One can use $J$-filter regular sequences to study the structure of local cohomology modules \cite{ASc,DQ,NS}.
\item Let $J = (f_1,...,f_r)$, where  $f_1,...,f_r$ is a $d$-sequence, i.e. $(f_1,...,f_{i-1}):f_if_t = (f_1,...,f_{i-1}):f_i$ for all $t \ge i$, and $i = 1,...,r$  \cite{Hu}. Then $f_1,...,f_r$ is a $J$-filter regular sequence \cite{Tr1}. A special case of $d$-sequence is the $p$-standard sequence, which was used to construct Macaulayfication of Noetherian schemes \cite{CC,Ka}.
\end{enumerate}
\end{example}

In the following we will prove Theorem \ref{intro}. 
By Lemma \ref{repre}, it suffices to show that $\ini(I) = \ini(I')$. 
This will be done by induction on the length of the filter-regular sequence $f_1,...,f_r$.
After proving the case $r=1$ we need to show that $f_1',f_2,...,f_r$ is also a filter-regular sequence with 
$\ini(I) = \ini(f_1',f_2,...,f_r)$. Then we can apply the induction hypothesis to the ring $R/(f_1')$ in order to obtain 
$\ini(I) = \ini(I')$.

First, we have to study the case $r = 1$.
For any $R$-module $M$, we define 
$$a_J(M) := \inf\{n|\  J^nM = 0\}.$$
Notice that $a_J(M) = \infty$ if $J^nM \neq 0$ for all $n \ge 0$. This notion is similar to the Loewy length of $M$, which is just $a_{\fm}(M)$. We call $a_J(M)$ the $J$-{\it Loewy length} of $M$. 
By this notation, $f$ is a regular element if and only if $a_J(0:f) = 0$ and $f$ is a
$J$-filter regular element if and only if $a_J(0:f) < \infty$. \par

\begin{proposition} \label{one}
Let $f$ be a $J$-filter regular element. Set 
$$c = \max\{a_J(0:f), \arn_J(f)+1\}.$$
Let $f' = f +\e$, where $\e$ is an arbitrary element in $J^c$. Then 
\begin{enumerate}[{\rm (i)}]
\item $f'$ is a $J$-filter regular element with $0: f' =  0:f$.
\item $\ini(f') = \ini(f)$.
\end{enumerate}
\end{proposition}

\begin{proof} 
We will first show (ii). For that it suffices to show that $(gf)^* = (gf')^*$ for any element $g \in R$. 
If $gf = 0$, $g \in 0:f$. Since $J^c(0:f) = 0$, $g \varepsilon = 0$. Hence, $gf' = g(f+ \e) = 0$. 
In this case, we have $(gf)^* = (gf')^* = 0$. 
Therefore, we may assume that $gf \neq 0$. \par

Let $n = o(gf)$. 
If $n \ge o(g \varepsilon)$, we have $n \ge o(\e) \ge c > \arn_J(f).$ Therefore,
$$gf \in (f) \cap J^n = J^{n-c+1} (J^{c-1} \cap (f)) = fJ^{n-c+1} (J^{c-1} : f).$$
It follows that $g \in J^{n-c+1} + (0:f)$. Hence, 
$$g\e \in \e J^{n-c+1} + \e (0:f) \subseteq J^{n+1} + J^c(0:f) = J^{n+1}.$$
Thus, $o(g\e) \ge n+1$, which contradicts the assumption $n \ge o (g \varepsilon)$. 
So we have $o(gf) = n < o (g\varepsilon)$, which implies  
$(gf')^* = (g(f+ \varepsilon))^* = (gf)^*$. \par

As a consequence, $(gf')^* = 0$ if and only if $(fg)^* = 0$.
This means $gf' = 0$ if and only if $gf = 0$. Therefore, $g \in 0:f'$ if and only if $g \in 0:f$.
So we obtain $0:f' = 0:f$. Since $f$ is $J$-filter regular, $a_J(0:f')  = a_J(0:f) \le \infty$.
Hence, $f'$ is also $J$-filter regular. So we obtain (i). 
\end{proof}

The next step is to replace $f_1,f_2$ by $f_1',f_2$ in the case $r=2$.
For that we need the following lemma. 

\begin{lemma}\label{exchange}
For any pair of elements $f_1, f_2 \in R$, there exists an isomorphism
$$\frac{(f_1):f_2}{(f_1) + 0:f_2} \cong \frac{(f_2):f_1}{(f_2) + 0:f_1}.$$
\end{lemma}

\begin{proof}
For every $x \in (f_1) : f_2$ we have $f_2x  = f_1y$ for some element $y \in (f_2):f_1$. 
Sending $x$ to $y$ induces a surjective map
$$\varphi:\, (f_1): f_2 \to \frac{(f_2):f_1}{(f_2) + 0:f_1}.$$
It is easy to check that $\mathrm{ker}(\varphi) = (f_1) + 0:f_2$, which implies the conclusion.
\end{proof}

\begin{proposition}\label{two}
Let $f_1,f_2$ be a $J$-filter regular sequence. Let $a_1 = a_J(0:f_1)$ and $a_2 = a_J((f_1):f_2/(f_1))$. Set
$$c = \max\{a_1+ a_2, \arn_J(f_1), \arn_J(f_1,f_2)\}+1$$
Let $f_1' = f_1 +\e$, where $\e$ is an arbitrary element in $J^c$. 
Then 
\begin{enumerate}[{\rm (i)}]
\item $f_1',f_2$ is a $J$-filter regular sequence
with $a_J(0:f_1') = a_1$ and 
$a_J((f_1'):f_2/(f_1')) \le 2a_2.$
\item $\ini(f_1',f_2) = \ini(f_1,f_2)$.
\end{enumerate} 
\end{proposition}

\begin{proof}
By Proposition \ref{one}(i), $f_1'$ is a $J$-filter regular element with $0:f_1' = 0:f_1$. Hence, $a_J(0:f_1') = a_1$.
To prove (i) it suffices to show that 
$a_J((f_1'):f_2/(f_1')) \le 2a_2$ because the finiteness of this invariant implies that $f_2$ is a $J$-filter regular element in $R/(f_1')$.\par

By the exact sequence
$$0 \to \frac{(f_1')+0:f_2}{(f_1')} \to \frac{(f_1'):f_2}{(f_1')} \to \frac{(f_1'):f_2}{(f_1')+0:f_2}  \to 0$$
we have
\begin{equation}
a_J((f_1'):f_2/(f_1')) \le a_J((f_1'):f_2/(f_1')+0:f_2) + a_J((f_1')+0:f_2/(f_1')).
\end{equation}

First, we have 
$$\frac{(f_1')+0:f_2}{(f_1')} \cong \frac{0:f_2}{(f_1') \cap (0:f_2)} \cong \frac{0:f_2}{f_1'(0:f_1'f_2)} 
= \frac{0:f_2}{f_1'(0:f_1f_2)} .$$
We will show that $f_1'(0:f_1f_2) = f_1(0:f_1f_2)$.
By Nakayama's lemma, it suffices to prove that  $f_1'(0:f_1f_2) \subseteq f_1(0:f_1f_2)$ and
$f_1(0:f_1f_2) \subseteq f_1'(0:f_1f_2) + \fm f_1(0:f_1f_2)$.
Since $f'_1 = f_1+\e$, we only need to show that  $\e(0:f_1f_2) \subseteq \fm f_1(0:f_1f_2)$.  
Note that $\e \in J^{a_1+a_2+1}$ and $J^{a_1}(0:f_1) = 0$. Then $\e(0:f_1f_2) \subseteq J^{a_2+1}(0:f_2)$.
We have
$$\frac{0:f_2}{f_1(0:f_1f_2)} =  \frac{0:f_2}{(f_1) \cap (0:f_2)} \cong \frac{(f_1)+0:f_2}{(f_1)} \subseteq \frac{(f_1):f_2}{(f_1)}.$$
Since $J^{a_2}((f_1):f_2/(f_1)) = 0$, $J^{a_2}(0:f_2/f_1(0:f_1f_2)) = 0$.
Hence, $J^{a_2+1}(0:f_2) \subseteq \fm f_1(0:f_1f_2)$. So we get $f_1'(0:f_1f_2) = f_1(0:f_1f_2)$. Therefore,
$$\frac{(f_1')+0:f_2}{(f_1')} \cong \frac{0:f_2}{f_1'(0:f_1'f_2)} = \frac{0:f_2}{f_1'(0:f_1f_2)} = \frac{0:f_2}{f_1(0:f_1f_2)}\subseteq \frac{(f_1):f_2}{(f_1)}.$$
From this it follows that 
\begin{equation}
a_J((f_1')+0:f_2/(f_1')) \le a_J((f_1):f_2/(f_1)) = a_2.
\end{equation}

By Lemma \ref{exchange}, 
$$
\frac{(f_1'):f_2}{(f_1')+0:f_2} \cong \frac{(f_2):f_1'}{(f_2)+0:f_1'} = \frac{(f_2):f_1'}{(f_2)+0:f_1}
$$
We shall see that $(f_2):f_1' = (f_2):f_1$.

Let $\bar R = R/(f_2)$. Since $J^{a_2}((f_1):f_2) \subseteq (f_1) \subseteq (f_1)+0:f_2$, we have 
$J^{a_2}((f_2):f_1) \subseteq (f_2) + 0:f_1$ by Lemma \ref{exchange}.
From this it follows that
$$J^{a_1+a_2}((f_2):f_1) \subseteq (f_2) + J^{a_1}(0:f_1) = (f_2).$$
Therefore, $f_1$ is a $J$-filter regular element in $\bar R$ with 
$$a_{J\bar R}(0_{\bar R}:f_1) = a_J((f_2):f_1/(f_2)) \le a_1+a_2.$$  
By Lemma \ref{qar}, $\arn_{J\bar R}(f_1\bar R) \le \arn_J(f_1,f_2)$.
These facts show that we can apply Proposition \ref{one} to the elements $f_1, f_1'$ in the ring $\bar R = R/(f_2)$. 

By Proposition \ref{one}(i), $(f_2):f_1' = (f_2):f_1$. Hence
$$\frac{(f_1'):f_2}{(f_1')+0:f_2} \cong \frac{(f_2):f_1}{(f_2)+0:f_1} \cong \frac{(f_1):f_2}{(f_1)+0:f_2},$$
which implies
\begin{equation}
a_J((f_1'):f_2/(f_1')+0:f_2) \le a_J((f_1):f_2/(f_1)) = a_2.
\end{equation}
Combining (3.1)-(3,3) we obtain $a_J((f_1'):f_2/(f_1')) \le 2a_2$, which completes the proof of (i).

By Proposition \ref{one}(ii) we have $\ini(f_1'\bar R) = \ini(f_1\bar R)$. By Lemma \ref{qin}, this implies
$$\ini(f_1',f_2)/\ini(f_2) = \ini(f_1,f_2)/\ini(f_2).$$ 
Therefore, $\ini(f_1',f_2) = \ini(f_1,f_2)$.
So we obtain (ii).
\end{proof}

Now we can use induction to study the perturbation of an arbitrary $J$-filter regular sequence. 
The outcome is the following Theorem \ref{main}, which is the main result of this paper. 
What we really have to prove is Theorem \ref{main}(ii). However, the induction hypothesis needs Theorem \ref{main}(i).

Theorem \ref{main}(iii) and Theorem \ref{main}(iv) are just consequences of Theorem \ref{main}(ii). They give affirmative answers to two problems raised by Ma, Quy and Smirnov. The first problem asks for an explicit upper bound for the perturbation index of $\gr_J(R/I)$ \cite[Question 1]{MQS}, while the second problem states that the Artin-Rees number $\arn_J(I)$ is preserved under small perturbation \cite[Paragraph before Corollary 3.6]{MQS}. 

\begin{theorem} \label{main}
Let $J$ be an arbitrary ideal. Let $I = (f_1,...,f_r)$, where $f_1,\ldots, f_r$ is a $J$-filter regular sequence. Let 
$$a_i = a_J((f_1,...,f_{i-1}):f_i/(f_1,...,f_{i-1}))$$
for $i = 1,...,r$. Set 
$$N = \max\{a_1+ 2a_2 +\cdots + 2^{r-1}a_r,\arn_J(f_1),...,\arn_J(f_1,...,f_r)\}+1.$$
Let $f_i' = f_i +\e_i$, where $\e_i$ is an arbitrary elements in $J^N$, $i = 1,...,r$, and $I' = (f_1',...,f_r')$.
Then \par
\begin{enumerate}[{\rm (i)}]
\item $f_1',\ldots, f_r'$ is a $J$-filter regular sequence with
$$a_J((f_1',...,f_{i-1}'):f_i'/(f_1',...,f_{i-1}')) \le 2^{i-1}a_i$$
for $i = 1,...,r.$
\item $\ini(I') = \ini(I)$.
\item $\gr_J(R/I') \cong \gr_J(R/I)$. 
\item $\arn_J(I') = \arn_J(I)$.
\end{enumerate}
\end{theorem}

\begin{proof}
We only need to prove (i) and (ii) because (iii) and (iv) follows from (ii) by Lemma \ref{repre} and Proposition \ref{ar}. 
The case $r = 1$ has been already proved in Proposition \ref{one}.

For $r \ge 2$, we will first show that 
$a_J((f_2,...,f_i):f_1/(f_2,...,f_i)) \le a_1 + \cdots + a_i$,
$i = 1,...,r$. Notice that $(f_2,...,f_i) = 0$ if $i = 1$.

The case $i = 2$ has been proved in the proof of Proposition \ref{two}.
For $i > 2$, we set $m = a_1+\cdots+a_{i-1}$. By induction we may assume that 
$a_J((f_2,...,f_{i-1}):f_1/(f_2,...,f_{i-1})) \le m.$
Then $J^m((f_2,...,f_{i-1}):f_1) \subseteq (f_2,...,f_{i-1}).$
By Lemma \ref{exchange} we have
$$\frac{(f_1,...,f_{i-1}):f_i}{(f_1,...,f_{i-1}) + (f_1,...,f_{i-2}):f_i} \cong \frac{(f_2,...,f_i):f_1}{(f_2,...,f_i) + (f_2,...,f_{i-1}):f_1}.$$
Since $J^{a_i}((f_1,...,f_{i-1}):f_i) \subseteq (f_1,...,f_{i-1})$, this implies
$$J^{a_i}((f_2,...,f_i):f_1) \subseteq (f_2,...,f_i) + (f_2,...,f_{i-1}):f_1.$$
Therefore,
$$J^{m+a_i}((f_2,...,f_i):f_1) \subseteq (f_2,...,f_i) + J^m((f_2,...,f_{i-1}):f_1) = (f_2,...,f_i),$$
which implies 
$$a_J((f_2,...,f_i):f_1/(f_2,...,f_i)) \le m+a_i = a_1 + \cdots +a_i.$$
As a consequence, $f_1$ is a $J$-filter regular element in $\bar R := R/(f_2,...,f_i)$.

Let $i = 1,...,r-1$. Since $f_{i+1}$ is a $J$-filter regular element in $R/(f_1,...,f_i)$, $f_1,f_{i+1}$ is a $J$-filter regular sequence in $\bar R$ with
$$a_J((f_1,...,f_i):f_{i+1}/f_1,...,f_i)) = a_{i+1}.$$
By Lemma \ref{qar}, $\arn_{J\bar R}(f_1\bar R) \le \arn_J(f_1,...,f_i)$ and $\arn_{J\bar R}((f_1,f_{i+1})\bar R) \le \arn_J(f_1,...,f_{i+1})$. Hence,
$$N \ge \max\{a_1+\cdots+a_{i+1}, \arn_{J\bar R}(f_1\bar R)),\arn_{J\bar R}((f_1,f_{i+1})\bar R)\}+1.$$
Applying Proposition \ref{two}(i) to the sequence $f_1,f_{i+1}$ in $\bar R$, we obtain that $f_{i+1}$ is a $J$-filter regular element in $R/(f_1',f_2,...,f_i)$ with
$$a_J((f_1',f_2,...,f_i):f_{i+1}/(f_1',f_2,...,f_i)) \le 2a_{i+1}.$$
From this it follows that $f_1',f_2,...,f_r$ is a $J$-filter regular sequence. 
By Proposition \ref{two}(ii), we have
$\ini((f_1',f_{i+1})\bar R) = \ini((f_1,f_{i+1})\bar R)$. 
Hence $\ini(f_1',f_2,...,f_{i+1}) = \ini(f_1,f_2,...,f_{i+1})$ by Lemma \ref{qar}.
By Proposition \ref{ar}, this implies 
$$\arn_J(f_1',f_2,...,f_{i+1}) = \arn_J(f_1,f_2,...,f_{i+1}).$$

Set $R' = R/(f_1')$. Using induction on $r$, we may assume that 
\begin{enumerate}[(i')]
\item $f_2',..., f_r'$ is a $J$-filter regular sequence in $R'$ with
$$a_J((f_2',...,f_{i-1}')R': f_i'/(f_1',...,f_{i-1}')R') = 2^{i-1}a_i$$
for $i = 2,...,r$. 
\item $\ini((f_2',...,f_r')R') = \ini((f_2,...,f_r)R')$.
\end{enumerate}

By Proposition \ref{one}(i), $f_1'$ is a $J$-filter regular element with $0:f_1' = 0:f_1$. Therefore,
(i') implies that $f_1',..., f_r'$ is a $J$-filter regular sequence with 
$$a_J((f_1',...,f_{i-1}'): f_i'/(f_1',...,f_{i-1}')) = 2^{i-1}a_i$$
for $i = 1,...,r.$ 
By Lemma \ref{qin}, (ii') implies
$$\ini(f_1',f_2',...,f_r') = \ini(f_1',f_2,...,f_r).$$

Set $S = R/(f_2,...,f_r)$. We have shown above that $f_1$ is a $J$-filter regular element in $S$ with
$a_J((f_2,...,f_r): f_1) \le a_1 +\cdots + a_r.$
By Lemma \ref{qar}, $\arn_J(f_1S) \le \arn_J(f_1,...,f_r)$.
Applying Proposition \ref{one}(ii) to this case, we get
$\ini(f_1'S) = \ini(f_1S)$. By Lemma \ref{qin}, this implies
$$\ini(f_1',f_2,...,f_r) = \ini(f_1,f_2,...,f_r).$$
So we obtain $\ini(f_1',f_2',...,f_r') = \ini(f_1,f_2,...,f_r).$
The proof is now complete.
\end{proof}

The upper bound for the perturbation number of Theorem \ref{main} becomes especially simple if the ideal is a complete intersection, when $a_1 = \cdots = a_r = 0$. Moreover, since regular sequences are permutable, we do not need Proposition \ref{two} and hence the numbers $\arn_J(f_1,...,f_i)$ for $i < r$, which gives the following bound.

\begin{corollary} \label{regular} {\bf \cite[Theorem 3.6]{Ch}}
Let $J$ be an arbitrary ideal. Let $I = (f_1,...,f_r)$, where $f_1,\ldots, f_r$ is a regular sequence. Set 
$N =  \arn_J(f_1,...,f_r)+1.$
Let $f_i' = f_i +\e_i$, where $\e_i$ is an arbitrary elements in $J^N$, $i = 1,...,r$, and $I' = (f_1',...,f_r')$.
Then $f_1',\ldots, f_r'$ is a regular sequence and
$\ini(I') = \ini(I)$. 
\end{corollary}

\begin{proof}
Note that an element $f$ is regular if and only if $a_J(0:f)= 0$.
Applying Proposition \ref{one} to the element $f_1$ in $R/(f_2,...,f_r)$, 
we get that $f_1'$ is a regular element in $R/(f_2,...,f_r)$ and $\ini(f_1',f_2,...,f_r) = \ini(f_1,...,f_r)$.
As a consequence, $f_2,\ldots, f_r,f_1'$ is a regular sequence and $\arn_J(f_1',f_2,...,f_r) = \arn_J(f_1,f_2,...,f_r)$ by Proposition \ref{ar}.
Using induction, we may assume that $f_2',...,f_r'$ is a regular sequence in $R/(f_1')$ and 
$\ini(f_1',f_2',...,f_r') =  \ini(f_1',f_2...,f_r)$.
Therefore, $\ini(I') = \ini(I)$. 
\end{proof}

We can not weaken the condition $f_1,...,f_r$ being a $J$-filter regular sequence in Theorem \ref{main}.
By Theorem \ref{main}(ii), this condition implies $\ini(f_1',...,f_i') = \ini(f_1,...,f_i)$ for any sequence $f_1',...,f_r'$ with $f'_i - f_i \in J^N$, $i = 1,...,r$. We shall see that the converse of this implication also holds. For this we need the following observation.

\begin{lemma} \label{filter}
Let $f$ be a $J$-filter regular element. Let $c = \arn_J(f)+1$. Then $f + \varepsilon$ is a $J$-filter regular element for every $\varepsilon \in J^c$.
\end{lemma}

\begin{proof} 
Suppose that $(0:f) \cap J^n = 0$. Let $g \in (0:(f+ \varepsilon)) \cap J^n$. Then 
$$gf = - g \varepsilon \in J^{n + c} = J^{n+1}(J^{c-1} \cap (f)) = J^{n+1}f(J^{c-1}:f).$$
Therefore, $g \in  J^{n+1}(J^{c-1}:f) + (0:f) \subseteq J^{n+1} + (0:f).$ Hence, 
$$g \in (J^{n+1} + (0:f)) \cap J^n = J^{n+1} + (0:f) \cap J^n = J^{n+1}.$$
Thus, $g \in (0:(f+ \varepsilon)) \cap J^{n+1}$. Proceeding as above, we get $g \in (0:(f+ \varepsilon)) \cap J^{n + m}$ for all $m \ge 0$. By Krull's intersection theorem, $g = 0$. Therefore, $(0:(f+ \varepsilon)) \cap J^n =0$, which implies that  $f + \varepsilon$ is a $J$-filter regular element.
\end{proof}

The following example shows that the condition $f \in J^c$ for $c = \arn_J(f)+1$ of Lemma \ref{filter} does not imply $\ini(f) = \ini(f+\e)$ as in Proposition \ref{one}(ii).

\begin{example} 
Let $k[[X,Y]]$ be a power series ring in two variables over a field $k$, 
$R = k[[X,Y]]/(XY,Y^4) = k[[x,y]]$, $J = \fm$, and $f = x$. 
Then $\arn_{\fm}(x) = 1$. Put $\e  = y^2$. 
It is easy to check that $\gr_\fm(R) = k[X,Y]/(XY,Y^4) = k[x^*,y^*]$, $\ini(x) = (x^*)$ and $\ini(x+y^2) = (x^*)+(y^*)^3$.
\end{example}

\begin{theorem} \label{converse}
Let $J$ be an arbitrary ideal and $f_1,...,f_r$ a sequence of elements in $R$.
There exists a number $N$ such that 
$$\ini(f_1',...,f_i') = \ini(f_1,...,f_i)$$
for any sequence $f_1',...,f_r'$ with  $f'_i - f_i \in J^N$, $i = 1,...,r$,
if and only if $f_1,...,f_r$ is a $J$-filter regular sequence.
\end{theorem}

\begin{proof} 
Assume that there exists a number $N$ such that $\ini(f_1',...,f_i') = \ini(f_1,...,f_i)$
if  $f'_i - f_i \in J^N$, $i = 1,...,r$. 
By \cite[Lemma of Appendix]{Da}, there exists $\e_i \in J^N$ such that
$f'_i = f_i+\e_i \not\in \fp$ for all associated primes $\fp \not\supseteq J$ of $(f_1,...,f_{i-1})$, $i = 1,...,r$.
Therefore, $f_i'$ is a $J$-filter regular element in $R/(f_1,...,f_{i-1})$.
Without restriction we may assume that $N \ge \arn_J(f_1,...,f_i)+1$.
Since $\ini(f_1,...,f_{i-1},f_i') = \ini(f_1,...,f_i)$, we have $\arn_J(f_1,...,f_{i-1},f_i') = \arn_J(f_1,...,,f_i) < N$.
Let $\bar R = R/(f_1,...,f_{i-1})$. By Lemma \ref{qar}, $\arn_J(f_i'\bar R) \le \arn_J(f_1,...,f_{i-1},f_i') < N$.
By Lemma \ref{filter} (applied to the element $f_i'$ in the ring $\bar R$), the condition $f_i - f_i' \in J^N$ implies that $f_i$ is a $J$-filter regular element in $\bar R$. This holds for all $i = 1,...,r$. 
Therefore, $f_1,...,f_r$ is a $J$-filter regular sequence. 
The converse of this conclusion follows from Theorem \ref{main}. 
\end{proof}

By Proposition \ref{ar}, $\gr_J(R/I) \cong \gr_J(R/I')$ if and only if $\ini(I) = \ini(I')$. Therefore, Theorem \ref{intro2} follows from Theorem \ref{converse}. In particular, we have the following criterion in the case $r = 1$, which shows that the condition of $I$ being generated by a $J$-filter regular sequence is almost a necessary and sufficient condition for $\gr_J(R/I) \cong \gr_J(R/I')$ under small $J$-adic perturbations $I'$ of $I$.

\begin{corollary} 
Let $J$ be an arbitrary ideal and $f$ an element in $R$. There exists a number $N$ such that $\gr_J(R/(f))\cong \gr_J(R/(f'))$
for any element $f'$ with  $f'- f \in J^N$ if and only if $f$ is a $J$-filter regular element.
\end{corollary}

Now we will give some non-trivial applications of Theorem \ref{main}. First, inspired by the conjecture of Srinivas and Trivedi on the invariance of the Hilbert-Samuel function under small perturbations we study 
the following generalization of the Hilbert-Samuel function for non-$\fm$-primary ideals.

Let $G := \gr_\fm(\gr_J(R)),$
which is a standard bigraded algebra with bigraded components
$$G_{uv}=\frac{\fm^uJ^v+J^{v+1}}{\fm^{u+1}J^v+J^{v+1}}.$$
for $u,v \ge 0$. Note that $G_{uv}$ is a module of finite length.
Let 
$$h(r,s)=\sum_{u=0}^{r} \sum_{v=0}^{s} \ell(G_{uv}),$$
where $\ell(\cdot)$ denotes length. This function was studied first by Achilles and Manaresi \cite{AM}.
For $r$ and $s$ sufficiently large, $h(r,s)$ is a polynomial of degree $d = \dim R$.
If one writes this polynomial in the form
$$ \sum_{i=0}^d \frac{c_i (J)}{(d-i)! \, i!} \, r^{d-i}s^{i}  +\mbox{  terms of lower degree,}
$$
then $c_0(J),...,c_d(J)$ are nonnegative integers. Achilles and Manaresi call them the multiplicity sequence of $R$ with respect to $J$. This notion plays a similar role as the multiplicity for non-$\fm$-primary ideals \cite{FOV,PTUV}. In particular, the Segre numbers in singularity theory \cite{GG} are a special case of the multiplicity sequence. We call $h(r,s)$ the {\it Achilles-Manaresi function} of $R$ with respect to $J$.

\begin{corollary} \label{AM}
Let $J, I, I'$ be as in Theorem \ref{main}. Then $R/I$ and $R/I'$ share the same Achilles-Manaresi function with respect to $J$. 
\end{corollary}

\begin{proof} 
By Theorem \ref{main}(iii), we have $\gr_\fm(\gr_J(R/I)) \cong \gr_\fm(\gr_J(R/I'))$, which implies the conclusion.
\end{proof}

Next we study the impact of perturbations to the Rees algebra $\Re_J(R/I)$.

\begin{corollary} \label{Rees}
Let $J, I, I'$ be as in Theorem \ref{main}.  Then the following invariants and properties of the Rees algebra $\Re_J(R/I)$ are preserved when passing to $\Re_J(R/I')$:
\begin{enumerate}[\rm (i)]
\item the relation type,
\item the Castelnuovo-Mumford regularity,
\item the Cohen-Macaulayness,
\item the Gorensteiness.
\end{enumerate}
\end{corollary}

\begin{proof} 
By Theorem \ref{main}(iii), $\gr_J(R/I') \cong \gr_J(R/I)$. 
It is well known that $J$ is generated by a regular sequence in $R/I$ or $R/I'$ if and only if 
$\gr_J(R/I)$ or $\gr_J(R/I')$ is isomorphic to a polynomial ring over $R/(I+J) = R/(I'+J)$. 
Therefore, $J$ is generated by a regular sequence in $R/I$ if and only if $J$ is generated by a regular sequence in $R/I'$.
If $J$ is generated by a regular sequence in $R/I$ and $R/I'$, we have
$\rel(\Re_J(R/I)) = \rel(\Re_J(R/I')) = 1$. 
If $J$ is not generated a regular sequence in $R/I$ and $R/I'$, then $\rel(\Re_J(R/I)) = \rel(\gr_J(R/I))$ and $\rel(\Re_J(R/I')) = \rel(\gr_J(R/I'))$ by  \cite[Proposition 3.3]{PV}. Since $\rel(\gr_J(R/I)) = \rel(\gr_J(R/I'))$, this implies $\rel(\Re_J(R/I)) = \rel(\Re_J(R/I'))$.

By \cite[Lemma 4.8]{O}, the associated graded ring and the Rees algebra have the same Castelnuovo-Mumford. Therefore,
$$\reg(\Re_J(R/I')) = \reg(\gr_J(R/I'))= \reg(\gr_J(R/I)) = \reg(\Re_J(R/I)).$$

By \cite[Theorem 1.1]{TI} and \cite[Theorem 1.1]{TVZ}, the Cohen-Macaulayness and the Gorensteiness of the Rees algebra can be characterized completely 
in terms of the local cohomology modules of the associated graded ring. Since the local cohomology modules of $\gr_J(R/I')$ and $\gr_J(R/I)$ are isomorphic, $\Re_J(R/I')$ is Cohen-Macaulay or Gorenstein if and only if so is $\Re_J(R/I)$.
\end{proof}

Recently, Duarte \cite{Du} studied the invariance of the Betti numbers of $R/I$ under small $\fm$-adic perturbations. It would be of interest to extend his study to arbitrary $J$-adic perturbations.


\section{Perturbation index and extended degree}

Throughout this section, let $(R,\fm)$ be a local ring and $J$ a $\fm$-primary ideal.
For any finitely generated $R$-module $M$ we denote by $e(J,M)$ 
the multiplicity of $M$ with respect to $J$. 

If $R$ is a Cohen-Macaulay ring and $I$ an ideal generated by a regular sequence, Trivedi \cite[Corollary 5]{Tri} gave an upper bound for the perturbation index of $\gr_\fm(R/I)$ in terms of $e(\fm,R/I)$. 
If $R$ is an arbitrary local ring, it is easy to see that there is no upper bound for the perturbation index of $\gr_\fm(R/I)$which depends only on $e(\fm,R/I)$. 

\begin{example} \label{ex} 
Let $k[[X,Y,Y]]$ be a power series ring in three variables over a field $k$, 
$R = k[[X,Y,Z]]/(XZ,YZ,Z^{n+2}) = k[[x,y,z]]$, $J = \fm$, and $f = x$. 
Then $R$ is a generalized Cohen-Macaulay ring, $x$ a filter regular element and $e(\fm,R/(x)) = 1$.
It is easy to check that $\gr_\fm(R) = k[X,Y,Z]/(XZ,YZ,Z^{n+2}) = k[x^*,y^*,z^*]$, $\ini(x) = (x^*)$ and $\ini(x+z^n) = (x^*)+(z^*)^{n+1}$. By Lemma \ref{repre},  $\gr_\fm(R/(x)) \cong \gr_\fm(R)/(x^*)$ and $\gr_\fm(R/(x+z^n)) \cong \gr_\fm(R)/((x^*)+(z^*)^{n+1})$. Therefore, there is a surjective map from $\gr_\fm(R/(x))$ to $\gr_\fm(R/(x+z^n))$, whose kernel is not zero in degree $\ge n+1$. From this it follows that the perturbation index of $\gr_\fm(R/(x))$ is at least $n+1$.
\end{example}

We can give an upper bound for the perturbation index of $\gr_J(R/I)$ in terms of the following generalization of the multiplicity,  which was introduced by Doering, Gunston and Vasconcelos \cite{DGV} for $J=\fm$ and by Linh \cite{Li} for any $\fm$-primary ideal $J$.

\begin{definition} \label{extended}
An {\em extended degree} is a number $D(J,M)$
assigned to every finitely generated $R$-module $M$ which satisfies the following conditions: 
\begin{enumerate}
\item $D(J,M) = D(J,M/L) + \ell(L)$, where $L$ is the maximal submodule of $M$ having
finite length,
\item $D(J,M)  \ge  D(J,M/xM)$ for a general element $x \in J$,
\item $D(J,M)= e(J,M)$ if $M$ is a Cohen-Macaulay module.
\end{enumerate}	
\end{definition}

\begin{remark}
A prototype of extended degree is the homological degree introduced by Vasconcelos \cite{Va2}. Let $R$ be a homomorphic image of a Gorenstein ring $S$. For a finitely generated $R$-module $M$ with $\dim M = d$, the {\em homological degree} of $M$ is defined inductively by setting $\hdeg(J,M) := \ell(M)$ if $d = 0$ and
$$\hdeg(J,M) := e(J,M) + \sum_{i=0}^{d-1} {d-1 \choose i} \hdeg(J,\Ext_S^{\dim S - i}(M,S))$$
if $d > 0$, where $e(J,M)$ denotes the multiplicity of $M$ with respect to $J$.
Note that $\dim (\Ext_S^{\dim S - i}(M,S)) < d$ for $i = 0,...,d-1$.
If $R$ is not a homomorphic image of a Gorenstein ring, we define
$$\hdeg(J,M):= \hdeg(J,M\otimes \hat R),$$
where $\hat R$ denotes the $\fm$-adic completion of $R$. In particular, if $M$ is a generalized Cohen-Macaulay module, we have
$$\hdeg(J,M) = e(J,M) + \sum_{i=0}^{d-1}{d-1 \choose i}\ell(H_{\fm}^i(M)),$$
where $H_{\fm}^i(M)$ denotes the $i$-th local cohomology module of $M$. \par
Among all extended degrees $D(J,M)$ there is a minimal one which is called the smallest extended degree $\sdeg(J,M)$. There is no explicit formula for $\sdeg(J,M)$. However, if $J = \fm$, this extended degree can be efficiently computed by using Gr\"obner basis technique \cite{NR}. 
\end{remark}

Our bound for the perturbation index is based on the following upper bound for the Castelnuovo-Mumford regularity of the associated graded ring in terms of an arbitrary extended degree. 

\begin{proposition} \label{linh} 
Let $J$ be an $\fm$-primary ideal of $R$. 
Let $M$ be a finitely generated $R$-module and $d := \dim(M)$. Then
$$\reg(\gr_J(M)) \le (2^{d!}-1)D(J,M)^{3d!-2}-1.$$
\end{proposition}

\begin{proof}
If $d = 0$, $M$ is of finite length and $\sum_{n\ge 0}\ell(J^nM/J^{n+1}M) = \ell(M)$.
Since $\gr_J(M)_n = J^nM/J^{n+1}M = 0$ for $n \gg 0$, $\reg(\gr_J(M))$ is the largest integer $n$ such that $\ell(J^nM/J^{n+1}M) \neq 0$. From this it follows that $\reg(\gr_J(M)) \le \ell(M) - 1$. By Definition \ref{extended}(1), $D(J,M) = \ell(M)$. Therefore, \par
\indent (i) $\reg(\gr_J(M)) \le D(J,M)-1$ if $d = 0$.\par
\noindent By \cite[Theorem 1.1]{Li}, we have \par
\indent  (ii) $\reg(\gr_J(M)) \le D(J,M)-1$ if $d = 1$,\par
\indent  (iii) $\reg(\gr_J(M)) \le 2^{(d-1)!}D(J,M)^{3(d-1)!-1} - 1$ if $d \ge 2$.\par
\noindent Note that $0! = 1$. Combining (i)-(iii) we obtain $\reg(\gr_J(M)) \le (2^{d!}-1)D(J,M)^{3d!-2}-1$ for all $d \ge 0$.
\end{proof}

\begin{theorem} \label{m-primary}
Let $(R,\fm)$ be an arbitrary local ring and $d = \dim R$.
Let $J$ be an arbitrary $\fm$-primary ideal of $R$. 
Let $I = (f_1,...,f_r)$, where $f_1,...,f_r$ is a filter-regular sequence in $R$. 
Let  $m = \max\{d!,r\}$ and $D := \max\{D(J, R/(f_1,...,f_i))|\ i = 0,...,r\}.$ Set
$N := 2^{m-1}D^{3d!-2}.$
Then $\gr_J(R/I) \cong \gr_J(R/I')$
for all ideals $I' = (f_1',...,f_r')$, where $f_1',...,f_r'$ is a sequence of elements in $R$ such that $f'_i - f_i \in J^N$, $i = 1,...,r$.
\end{theorem}

\begin{proof}
Since $f_1,...,f_r$ is a filter-regular sequence, $(f_1,...,f_{i-1}):f_i/(f_1,...,f_{i-1})$ is of finite length, $i = 1,...,r$. Hence,
\begin{align*}
a_J((f_1,...,f_{i-1}):f_i/(f_1,...,f_{i-1})) & \le a_{\fm}((f_1,...,f_{i-1}):f_i/(f_1,...,f_{i-1}))\\
& \le \ell((f_1,...,f_{i-1}):f_i/(f_1,...,f_{i-1})).
\end{align*}
By Definition \ref{extended}(1), 
$$\ell((f_1,...,f_{i-1}):f_i/(f_1,...,f_{i-1})) \le  D(J,R/(f_1,...,f_{i-1})) \le D.$$
From this it follows that
$$\sum_{i=1}^r2^{i-1}a_J((f_1,...,f_{i-1}):f_i/(f_1,...,f_{i-1})) \le (1+2+\cdots+2^{r-1})D = (2^r-1)D \le N.$$
Note that $\dim R/(f_1,...,f_i) \le d$, $i = 1,...,r$. By Corollary \ref{reg} and Proposition \ref{linh},
$$\arn_J((f_1,...,f_i)) \le \reg(\gr_J(R/(f_1,...,f_i))+1 \le (2^{d!}-1)D(J,R/(f_1,...,f_i))^{3d!-2}-1.$$
Therefore,
$$\max\{\arn_J(R/(f_1)+1,...,\arn_J(R/(f_1,...,f_r)+1\} \le (2^{d!}-1)D^{3d!-2}-1 \le N.$$
Now we only need to apply Theorem \ref{main}(iii) to obtain the conclusion.
\end{proof}

\begin{remark}
The bound for the perturbation index of $\gr_J(R/I)$ in Theorem \ref{m-primary}  is far from the best possible as one can see from the proof. We do not know whether the complexity for such a bound is polynomial in $D$ or not. If $J = \fm$, one can use \cite[Theorem 3.3]{RTV} to get a slightly better bound for the perturbation index.
\end{remark}

If $\gr_J(R/I) \cong \gr_J(R/I')$, then 
$\ell(I+J^n/I+J^{n+1}) = \ell(I'+J^n/I'+J^{n+1})$
for all $n \ge 0$. Hence $R/I$ and $R/I'$ share the same Hilbert-Samuel function: 
$$\ell(R/I+J^n) = \ell(R/I'+J^n).$$ 
Thus, the Hilbert perturbation index is always less than or equal to the perturbation index of $\gr_J(R/I)$.  

We can derive from the Hilbert perturbation index an upper bound for the perturbation index of $\gr_J(R/I)$ by involving the Artin-Rees number.

\begin{proposition} \label{Hilbert}
Let $(R,\fm)$ be an arbitrary local ring and $J$ an $\fm$-primary ideal of $R$. 
Let $I = (f_1,...,f_r)$, where $f_1,...,f_r$ is a filter-regular sequence in $R$. Let $N = \max\{p,\arn_J(I)+1\}$, where $p$ is the Hilbert perturbation index of $R/I$ with respect to $J$. 
Then 
$\gr_J(R/I) \cong \gr_J(R/I')$
for all ideals $I' = (f_1',...,f_r')$, where $f_1',...,f_r'$ is a sequence of elements in $R$ such that $f'_i - f_i \in J^N$, $i = 1,...,r$.
\end{proposition}

\begin{proof}
Assume that $f'_i - f_i \in J^{\arn_J(I)+1}$, $i = 1,...,r$. Then there exists an epimorphism $\varphi$ from $\gr_J(R/I)$ to $\gr_J(R/I')$ \cite[Lemma 3.2]{MQS}. From this it follows that 
$$\ell(I+J^n/I+J^{n+1}) \ge \ell(I'+J^n/I'+J^{n+1})$$
for all $n \ge 0$. Therefore, $\ell(R/I+J^n) \ge \ell(R/I'+J^n)$. Now, it is clear that $\varphi$ becomes an isomorphism if $\ell(R/I+J^n) = \ell(R/I'+J^n)$ for all $n \ge 0$, which holds if $f'_i - f_i \in J^m$, $i = 1,...,r$.
\end{proof}

If $R$ is a  Cohen-Macaulay ring, $J = \fm$, and $f_1,...,f_r$ is part of a system of parameters, 
there is a linear bound for the Hilbert perturbation index in terms of $e(\fm,R/I)$ found by Srinivas and Trivedi \cite[Proposition 1]{ST2}. This bound was used by Trivedi \cite[Corollary 5]{Tri} to give a bound for the perturbation index of $\gr_J(R/I)$. 
Recently, Quy and V.D. Trung \cite[Theorem 1.3]{QT} have been able to extend 
the linear bound of Srinivas and Trivedi for the Hilbert perturbation index to generalized Cohen-Macaulay rings. Using this bound, we obtain the following 
explicit upper bound for the perturbation index of $\gr_J(R/I)$.

\begin{corollary} 
Let $(R,\fm)$ be a generalized Cohen-Macaulay ring and $d=\dim R$. 
Let $I = (f_1,...,f_r)$, where $f_1,...,f_r$ is part of a system of parameters in $R$, and $s = d-r$. Set 
$$N := (2^{s!}+1)\hdeg(\fm,R/I)^{3s!-2} + 1.$$
Then 
$\gr_\fm(R/I) \cong \gr_\fm(R/I')$
for all ideals $I' = (f_1',...,f_r')$, where $f_1',...,f_r'$ is a sequence of elements in $R$ such that $f'_i - f_i \in J^N$, $i = 1,...,r$.
\end{corollary}

\begin{proof}
Let $p$ be the Hilbert perturbation index of $R/I$ with respect to $\fm$.
By \cite[Theorem 1.3]{QT}, 
$$p \le s!\hdeg(\fm,R/I) + (s+1)\sum_{i=0}^{d-1}{d-1 \choose i}\ell(H_{\fm}^i(R))+1.$$
By the definition of the homological degree, 
$\displaystyle \sum_{i=0}^{d-1}{d-1 \choose i}\ell(H_{\fm}^i(R)) \le \hdeg(\fm,R/I)$. Hence, 
$$p \le (s! + s +1)\hdeg(\fm,R/I) +1 \le (2^{s!}+1)\hdeg(\fm,R/I) + 1.$$
By Corollary \ref{reg} and Proposition \ref{linh}, we have
$$\arn_\fm(I) \le \reg(\gr_\fm(R/I))+1 \le (2^{s!}-1)\hdeg(\fm,R/I)^{3s!-2}.$$ 
Hence, $\max\{p,\arn_J(I)+1\} \le (2^{s!}+1)\hdeg(\fm,R/I)^{3s!-2} + 1$.
Therefore, the conclusion follows from Proposition \ref{Hilbert}.
\end{proof}

Inspired of the results of the linear bounds in \cite[Proposition 1]{ST2} and \cite[Theorem 1.3]{QT} we raise the following questions, to which we are unable to give an answer.

\begin{question}
Let $(R,\fm)$ be an arbitrary local ring and $J$ an $\fm$-primary ideal. Let $I = (f_1,...,f_r)$, where $f_1,...,f_r$ is a filter-regular sequence. Does there exist a linear upper bound for the Hilbert perturbation index of $R/I$ with respect to $J$ in terms of the extended degree $D(J,R/I)$?
\end{question}

\begin{question}
Is the perturbation index of $\gr_J(R/I)$ equal to the Hilbert perturbation index of $R/I$ with respect to $J$?
\end{question}


\section{Perturbation with respect to a filtration}

Let $R$ be a convergent power series ring over the field $\Bbb R$ or $\Bbb C$ or a power series ring in several variables over a field. 
Let $<$ be a {\em monomial order}, i.e. a total order on the set of monomials of $R$ that satisfies the conditions: $1 < g$  for all monomials $g \in R$ and $g_1 <  g_2$ implies $hg_1 < hg_2$ for all monomials $g_1,g_2, h \in R$ \cite{Be}. 
For every power series $f \in R$, let $\ini_<(f)$ denote the smallest monomial of $f$ with respect to $<$.
Let $I$ be an ideal of $R$. The initial ideal $\ini_<(I)$ of $I$ is the ideal generated by the initial monomials $\ini_<(f)$, $f \in I$. Let $j^n(f)$ denotes the the $n$-jet of $f$, i.e. the polynomial part of degree $n$ of $f$.

This  section is motivated by the following problem.

\begin{problem} \label{adamus} 
Let $I = (f_1,...,f_r)$, where $f_1,...,f_r$ is a regular sequence in $R$. Let $<$ be an arbitrary monomial order. Does there exist a number $N$ such that for $n \ge N$, $j^n(f_1),...,j^n(f_n)$ is a regular sequence and $\ini_<(I) = \ini_<(I')$, where $I' = (j^n(f_1),...,j^n(f_n))$?
\end{problem}

If $<$ is the degree lexicographic monomial order, this was a conjecture of Adamus and Syedinejad \cite[Conjecture 3.7]{AS}. A solution of this conjecture can be deduced from the invariance of Hilbert-Samuel functions under small perturbations, which was proved by Srinivas and Trivedi \cite{ST1}.
In fact, it is well known that the Hilbert-Samuel function of $R/I$ is the same as that of $R/\ini(I)$. 
By \cite[Lemma 3.2(i)]{AS}, $\ini_<(I) \subseteq \ini_<(I')$ for $n \gg 0$. Therefore, if $R/I$ and $R/I'$ share the same Hilbert-Samuel function, then $R/\ini_<(I)$ and $R/\ini_<(I')$ share the same Hilbert-Samuel function, which implies $\ini_<(I) = \ini_<(I')$ for $n \gg 0$. A direct proof of this conjecture were given by Adamus and Patel \cite{AP1,AP2}, also by using Hilbert-Samuel functions.  
For an arbitrary monomial order $<$, one can not solved Problem \ref{adamus}  by means of  the invariance of Hilbert-Samuel functions under small perturbations. In fact, $R/I$ and $R/\ini_<(I)$ need not have the same Hilbert function because a monomial order need not be a refinement of the degree order. 

We shall give a positive answer to Problem \ref{adamus}  for any Noetherian monomial order. 
Recall that a monomial order  $<$ is {\em Noetherian} if for every monomial $f$ there are only finitely many monomials $g$ with $g<f$. We can view the ideal $\ini_<(I)$ with respect to a Noetherian monomial order as the initial ideal of $I$ with respect to a filtration of ideals, and we will show that the initial ideal with respect to such a filtration does not change under small perturbations.

From now on, let $R$ be a local ring. A {\em filtration} $F$ of ideals in $R$ is a sequence of ideals $\{J_n\}_{n\ge 0}$ which satisfy the following conditions for all $m,n \ge 0$:
\begin{enumerate}
\item $J_0 = R$,
\item $J_m \subseteq J_n$ if $m > n$,
\item $J_mJ_n \subseteq J_{m+n}$.
\end{enumerate}

A filtration $F = \{J_n\}_{n\ge 0}$ is {\em Noetherian} if the Rees algebra 
$\Re_F(R) := \bigoplus_{n\ge 0}J_n$ is Noetherian or, equivalently, if there exists a number $c$ such that $J_{n+1} = \sum_{i=1}^cJ_iJ_{c-i}$ for all $n \ge c$ \cite[Proposition 5.4.3]{BH}.
We can approximate a Noetherian filtration by an adic filtration. 

\begin{lemma} \label{J1}
Let $F = \{J_n\}_{n\ge 0}$ be a Noetherian filtration of ideals. Let $\D$ be the maximum degree of the elements of a minimal homogeneous generating set of $\Re_F(R)$. Then $J_{n\D} \subseteq J_1^n$ for all $n \ge 0$. 
\end{lemma}

\begin{proof}
The case $n = 0$ is trivial. For $n > 0$ we may assume that $J_{(n-1)\D} \subseteq J_1^{n-1}$. 
Since $\Re_F(R)$ is generated by homogeneous elements of degree $\le \D$, $J_{n\D} = \sum_{i=1}^\D J_{n\D-i}J_i$
for all $n \ge 1$. Note that $J_{n\D-i} \subseteq J_{(n-1)\D}$ and $J_i \subseteq  J_1$, $i = 1,...,\D$. Then 
$J_{n\D} \subseteq J_{(n-1)\D}J_1 \subseteq J_1^n.$
\end{proof}

To each filtration $F$ one attaches the associated graded ring
$$\gr_F(R) := \bigoplus_{n \ge 0}J_n/J_{n+1}.$$
If $F$ is a Noetherian filtration, it follows from Lemma \ref{J1} and Krull's intersection theorem that
 $\bigcap_{n \ge 0}J_n = 0.$
Therefore, if $f \neq 0$, there is a unique number $n$ such that $f \in J_n \setminus J_{n+1}$.
In this case, one defines the initial element of $f$ as the residue class $f^*$ of $f$ in $J_n/J_{n+1} \subset \gr_F(R)$. For convenience, we set $0^* = 0$.

\begin{remark} \label{filtration}
For a Noetherian monomial order $<$ in a power series ring $R$,
we can list the monomials in an increasing sequence $1 = g_0 < g_1 < g_2 < \cdots.$ 
Let $J_n$ be the ideals generated by all monomials $\ge g_n$. 
Then $F := \{J_n\}_{n\ge 0}$ is a filtration of ideals. In fact, since $1 = g_0 < g_1 < \cdots < g_m = g_mg_0 < g_mg_1 < \cdots < g_mg_n$, we have $g_{m+n} < g_mg_n$ for all $m,n \ge 0$.
It is clear that $\gr_F(R/I)$ is isomorphic to the polynomial ring in the same number of variables as $R$. 
Let $\fm$ be the maximal ideal of $R$.
For each $n \ge 0$, there exist numbers $t$ and $m$ such that $F_n \subseteq \fm^t \subseteq F_m$. 
Therefore, for all ideals $I$ of $R$,
$$\bigcap_{n \ge 0}(J_n+I) = \bigcap_{t \ge 0}(\fm^t+I) = I$$
by Krull's intersection theorem. By \cite[Proposition 5.4.5]{BH}, this together with the Noetherian property of $\gr_F(R/I)$ imply that $F$ is a Noetherian filtration. Now, for every power series $f$, we can define the initial element of $f$ with respect to $F$, which is clearly the initial monomial of $f$ with respect to $<$.
\end{remark}

Let $I$ be an arbitrary ideal of the local ring $R$. We denote by $\ini_F(I)$ the {\em initial ideal} of $I$ in $\gr_F(R)$ generated by the initial elements $f^*$, $f \in I$. Set
$$\gr_F(R/I) := \bigoplus_{n \ge 0}(J_n+I)/(J_{n+1}+I),$$
which is the associated graded ring of $R/I$ with respect to the induced filtration of $F$ in $R/I$. Similarly as Lemma \ref{repre}, we have the following relationship between $\gr_F(R/I)$ and $\ini_F(I)$.

\begin{lemma}
Let $F$ be a Noetherian filtration of ideals. Then
$$\gr_F(R/I) = \gr_F(R)/\ini_F(I).$$
\end{lemma}

Inspired by Theorem \ref{main}(ii) we will study the invariance of the initial ideal $\ini_F(I)$ under small perturbations with respect to the filtration $F$. For that we need to extend the notion of Artin-Rees number.

If $F$ is a Noetherian filtration, the ideal $Q := \bigoplus_{n\ge 0}J_n \cap I$ of $\Re_F(R)$ is finitely generated. Let $c = d(Q)$, the maximum degree of the elements of a homogeneous minimal generating set of $Q$. It is easy to see that
$$J_n \cap I = \sum_{t=1}^cJ_{n-t}(J_t \cap I)$$
for all $n \ge c$. We call $d(Q)$ the {\em Artin-Rees number} of $I$ with respect to $F$, and we denote it by $\arn_F(I)$. 

We have the following formula for $\arn_F(I)$ in terms of the initial ideal $\ini_F(I)$.  

\begin{proposition} \label{ar2}
Let $F$ be a Noetherian filtration of ideals. Then 
$\arn_F(I) = d(\ini_F(I)).$
\end{proposition}

\begin{proof} 
Note that $d(\ini(I))$ is the least number $c$ such that for all $n \ge c$, $\ini(I)_n$ is generated by graded elements of $\ini(I)$ of degree $\le c$. Since 
$$\ini(I)_n = (J_n \cap I+J_{n+1})/J_{n+1},$$
it follows that $d(\ini(I))$ is the least integer $c$ such that
$$J_n \cap I + J_{n+1} = \sum_{t = 1}^cJ_{n-t}(J_t\cap I) + J_{n+1}$$
for all $n \ge c$. For $a =  \arn_F(I)$, we have 
$J_n \cap I = \sum_{t = 1}^aJ_{n-t}(J_t\cap I)$
for $n \ge a$. Hence, $d(\ini(I)) \le \arn_F(I)$.
\par

To prove $\arn_F(I) \le d(\ini(I))$ let $f$ be an arbitrary element of  $J_n \cap I$, $n \ge d(\ini(I))$.
Then $f \in \sum_{t = 1}^cJ_{n-t}(J_t\cap I) + J_{n+1}$.
Write $f = g +  f_1$ for some $g \in \sum_{t = 1}^cJ_{n-t}(J_t\cap I)$ and $f_1 \in J_{n+1}$. 
Since $f, g \in I$, we have 
$$f_1 \in J_{n+1} \cap I \subseteq \sum_{t = 1}^cJ_{n+1-t}(J_t\cap I) + J_{n+2} \subseteq \sum_{t = 1}^cJ_{n-t}(J_t\cap I) + J_{n+2}.$$ 
Hence, $f = g + f_1 \in \sum_{t = 1}^cJ_{n-t}(J_t\cap I) + J_{n+2}.$ 
Continuing like that we have $f \in \sum_{t = 1}^cJ_{n-t}(J_t\cap I) + J_m$ for all $m \ge n$. 
By Lemma \ref{J1}, there exists $\D$ such that $J_1^{m\D} \subseteq J_{m\D} \subseteq J_1^m$ for $m \ge 1$. Therefore, $f \in \sum_{t = 1}^cJ_{n-t}(J_t\cap I) + J_1^m$ for all $m \ge n$.
By Krull's intersection theorem, this implies $f \in \sum_{t = 1}^cJ_{n-t}(J_t\cap I)$.
Thus, $J_n \cap I \subseteq \sum_{t = 1}^cJ_{n-t}(J_t\cap I)$. Clearly, $J_n \cap I \supseteq \sum_{t = 1}^cJ_{n-t}(J_t\cap I)$.
So we have 
$$J_n \cap I = \sum_{t = 1}^cJ_{n-t}(J_t\cap I)$$ 
for $n \ge d(\ini(I))$, which implies $\arn_F(I) \le d(\ini(I))$.
\end{proof}

The above properties of the initial ideal allow us to prove a similar result as Theorem \ref{main} for perturbations with respect to a Noetherian filtration. 

\begin{theorem} \label{main2}
Let $F = \{J_n\}_{n\ge 0}$ be a Noetherian filtration of ideals in a local ring $R$. Let $\D$ be the maximum degree of the element of a minimal homogeneous generating set of the Rees algebra $\Re_F(R)$. Let $I = (f_1,...,f_r)$, where $f_1,\ldots, f_r$ is a $J_1$-filter regular sequence in $R$. Set
$$a_i = a_{J_1}((f_1,...,f_{i-1}):f_i/(f_1,...,f_{i-1}))$$
for $i = 1,...,r$, and 
$$N := \max\{(a_1+ 2a_2 +\cdots + 2^{r-1}a_r+1)\D,\arn_{F}(f_1)+1,...,\arn_{F}(f_1,...,f_r)+1\}.$$
Let $f_i' = f_i +\e_i$, where $\e_i$ is an arbitrary elements in $J_{N}$, $i = 1,...,r$, and $I' = (f_1',...,f_r')$.
Then 
$f_1',\ldots, f_r'$ is a $J_1$-filter regular sequence
and $\ini_F(I') = \ini_F(I)$.
\end{theorem}

The appearance of $\D$ in Theorem \ref{main2} comes from the approximation of $F$ by the  $J_1$-adic filtration. In fact, if $\e \in J_N$, then $\e \in J_1^{a_1+ 2a_2 +\cdots + 2^{r-1}a_r+1}$ by Corollary \ref{J1}.
Similar to Corollary \ref{regular} we have the following consequence.

\begin{corollary} \label{regular 2}
Let $F = \{J_n\}_{n\ge 0}$ be a Noetherian filtration of ideals in a local ring $R$. 
Let $I = (f_1,...,f_r)$, where $f_1,\ldots, f_r$ is a regular sequence in $R$. Set
$N =  \arn_{F}(f_1,...,f_r)+1.$
Let $f_i' = f_i +\e_i$, where $\e_i$ is an arbitrary elements in $J_{N}$, $i = 1,...,r$, and $I' = (f_1',...,f_r')$.
Then $f_1',\ldots, f_r'$ is a regular sequence and $\ini_F(I') = \ini_F(I)$.
\end{corollary}

Theorem \ref{filtration} has the following converse, which is similar to Theorem \ref{converse}.

\begin{theorem} \label{converse 2}
Let $F = \{J_n\}_{n\ge 0}$ be a Noetherian filtration of ideals in a local ring $R$. Let $f_1,...,f_r$ be elements in $R$.
There exists a number $N$ such that 
$$\ini_F(f_1',...,f_i') = \ini_F(f_1,...,f_i)$$
for any sequence $f_1',...,f_i'$ with  $f'_i - f_i \in J_N$, $i = 1,...,r$,
if and only if $f_1,...,f_r$ is a $J_1$-filter regular sequence.
\end{theorem}

We leave the reader to check the proofs of the above results following the arguments for similar results in Sections 2 and 3. 

Applying Corollary \ref{regular 2} to the filtration described in Remark \ref{filtration} we immediately obtain the following solution to Problem \ref{adamus}.

\begin{corollary}  
Let $R$ be a power series ring over a field or a convergent power series ring over $\Bbb R$ or $\Bbb C$ in several variables. Let $I = (f_1,...,f_r)$, where $f_1,\ldots, f_r$ is a regular sequence in $R$. For any Noetherian monomial order in $R$, there exist a number $N$ such that for $n \ge N$, $j^n(f_1),...,j^n(f_r)$ is a regular sequence and $\ini_<(I) = \ini_<(I')$, where $I' = (j^n(f_1),...,j^n(f_r))$. 
\end{corollary}



\end{document}